\documentclass[12pt]{amsart}

\usepackage{amssymb, amsthm, amscd}
\usepackage{latexsym}
\usepackage{amssymb}
\usepackage[T1]{fontenc}	
\usepackage[english]{babel}
\usepackage{amsmath,amssymb}
\usepackage{amsfonts}
\usepackage{amscd}
\usepackage{cite}
\usepackage{marvosym}
\usepackage{mathrsfs}
\usepackage{amsthm}

\newtheorem{theorem}{Theorem}[section]
\newtheorem*{thm}{Theorem}
\newtheorem*{theorem1}{Theorem 1}
\newtheorem*{theorem2}{Theorem 2}

\newtheorem{lemma}[theorem]{Lemma}
\newtheorem{proposition}[theorem]{Proposition}

\theoremstyle{definition}
\newtheorem{definition}[theorem]{Definition}

\theoremstyle{remark}
\newtheorem{remark}[theorem]{Remark}

\numberwithin{equation}{section}

\newcommand{\R}{\mathbb{R}}

\newcommand{\e}{e_\alpha}

\newcommand{\z}{\zeta}

\newcommand{\be}{\begin{equation}}
\newcommand{\ee}{\end{equation}}
\newcommand{\bd}{\begin{displaymath}}
\newcommand{\ed}{\end{displaymath}}
\newcommand{\del}{\delta}

\renewcommand{\b }{\beta }
\renewcommand{\d}{\delta }

\renewcommand{\e}{\varepsilon}

\renewcommand{\b }{\beta }
\renewcommand{\d}{\delta }
\newcommand{\D }{\Delta }

\renewcommand{\l }{\lambda }

\begin{document}
\date{}
\title[Rigidity of stable marginally outer trapped surfaces in initial data sets]
{Rigidity of stable marginally outer trapped surfaces in initial data sets}

\author{Alessandro Carlotto}
\address{ETH - Institute for Theoretical Studies  \\
	ETH \\
	Z\"urich, Switzerland}
\email{alessandro.carlotto@eth-its.ethz.ch}

\thanks{The author was partially supported by Stanford University and NSF grant DMS-1105323}

\begin{abstract} In this article we investigate the restrictions imposed by the dominant energy condition (DEC) on the topology and conformal type of \textsl{possibly non-compact} marginally outer trapped surfaces (thus extending Hawking's classical theorem on the topology of black holes).
We first prove that an unbounded, stable marginally outer trapped surface in an initial data set $(M,g,k)$ obeying the dominant energy condition is conformally diffeomorphic to either the plane $\mathbb{C}$ or to the cylinder $\mathbb{A}$ and in the latter case infinitesimal rigidity holds. As a corollary, when the DEC holds strictly this rules out the existence of trapped regions with cylindrical boundary. 
In the second part of the article, we restrict our attention to asymptotically flat data $(M,g,k)$ and show that, in that setting, the existence of an unbounded, stable marginally outer trapped surface essentially never occurs unless in a very specific case, since it would force an isometric embedding of $(M,g,k)$ into the Minkowski spacetime as a space-like slice.
\end{abstract}

\maketitle
\section{Introduction}

In General Relativity, the existence of a closed trapped surface in a space-like slice $(M,g,k)$ of a spacetime $(\mathbb{L},\gamma)$ is (under certain natural energy and causal conditions) symptomatic of the geodesic incompletenss of the spacetime in question. In physical terms, that spacetime must contain a black hole. However, when considering \textsl{marginally} outer trapped surfaces the a priori restriction to closed submanifolds is no longer completely justified, at least for very general classes of data. Indeed, when one considers MOTS as separating elements, namely as boundaries of a trapped region it seems conceivable that complete, unbounded MOTS may arise. To be more specific, let us recall here that given an initial data set $(M,g,k)$ the outer trapped region is the union of all domains bounded by weakly outer trapped surfaces and the, possibly empty, interior boundary of $M$: it was recently proven by Andersson and Metzger (Theorem 1.3 in \cite{AM09}) that the boundary of the trapped region is a smooth, embedded, outermost MOTS (in fact the only one). In their work, they assumed to deal with a \textsl{compact} ambient manifold with two closed boundary components $\partial^{+}M$ and $\partial^{-}M$ so that the boundary of the trapped region had itself to be closed, hence only consiting of spherical and (possibly exceptional) toroidal components by Hawking's theorem on the topology of black holes (see \cite{Haw72} and \cite{Gal08} for the associated rigidity phenomena). But when the existence assumption on $\partial^{+}M$ is dropped and $M$ is not compact, then it is a priori possible to deal with unbounded trapped regions for which the \textsl{conclusion} of the theorem by Andersson and Metzger still holds. Hence it becomes relevant to extend the aforementioned topological results to the case of \textsl{complete} (unbounded) stable MOTS. Incidentally, we observe here that while the notion of \textsl{outward} for a non-compact orientable surface is arbitrary, this is not the case for MOTS that arise as boundaries of the trapped region. 

A second very good reason to pursue our study is given by the recent work by Eichmair on the Plateau problem for marginally outer trapped surfaces \cite{Eic09}, as complete MOTS naturally arise as limits of MOTS with boundary (because of their curvature and area estimates, the latter being related to their $\lambda$-minimizing properties).  

We first prove that in general initial data sets the sole assumption of dominant energy condition forces severe restrictions on the conformal class of a complete, stable MOTS: in fact when $\mu>|J|_{g}$ strictly that rules out the existence of all types of MOTS but spherical and planar ones, namely those equivalent to $\mathbb{C}$.
  
\begin{theorem1}\label{thm:structconf}
Let $(M,g,k)$ be an initial data set of dimension three satisfying the dominant energy condition and let $\Sigma$ be a complete, two-sided \textsl{stable} MOTS in $M$. Then the following statements hold:
\begin{enumerate}
\item{If $\Sigma$ is compact, then it is conformally equivalent to the sphere $\mathbb{S}^{2}$ or to the torus $\mathbb{T}^{2}$. Moreover, in the latter case $\Sigma$ is flat, totally geodesic and can be embedded in a smooth local foliation $\left\{\Sigma_{t}\right\}_{t\in\left(-\e,\e\right)}$ where each leaf is itself a MOTS. As a result, if $\Sigma$ is outermost then it is conformally equivalent to $\mathbb{S}^{2}$.}
\item{If $\Sigma$ is not compact, then it is conformally equivalent to the complex plane $\mathbb{C}$ or to the cylinder $\mathbb{A}$. Moreover, in the latter case $\Sigma$ is infinitesimally rigid namely
\[
 K=\chi=\mu+J(\nu)=0 \ \textrm{identically on} \ \Sigma.
\]
 If the \textsl{strict} dominant energy condition holds, then only the first alternative can happen (and thus $\Sigma$ is conformally equivalent to $\mathbb{C}$).}
\end{enumerate} 
\end{theorem1}

The statements collected in (1) are well-known and date back to Hawking(in \cite{Haw72}, see also \cite{GS06} for the higher dimensional counterpart), and to Galloway \cite{Gal08} for the part concerning the construction of the foliation $\left\{\Sigma_{t}\right\}_{t\in\left(-\e,\e\right)}$: they have been stated here for completeness. Various comments related to Theorem \ref{thm:structconf} are appropriate.
First of all, the proof of the rigidity statement given in part (2) would be relatively simple if one made the assumption that the MOTS $\Sigma$ has finite total curvature (in which case one could easily adapt the argument given, for minimal surfaces, by Fischer-Colbrie and Schoen in \cite{FS80}) or if, alternatively, one assumed the sectional curvature of $(M,g)$ to be bounded (in which case one could combine the argument given by Schoen and Yau in \cite{SY82} with a preliminary deformation by means of Shi's complete Ricci flow). Instead, we do not make any such assumption here and thus the proof of Theorem \ref{thm:structconf} requires a combination of various ideas. We also emphasize that in part (2) of this theorem, the surface $\Sigma$ is not required to have quadratic area growth (nor to be embedded) and in this generality a rigorous argument is more subtle than it may look.
 
We later specify our study to asymptotically flat data and show that, in a wide setting, unbounded stable MOTS do not exist at all unless the spacetime in question is Minkowskian.   

\begin{theorem2}\label{thm:MOTSrig3}
Let $(M,g,k)$ be an initial data set of dimension three in boosted harmonic asymptotics. If it contains a complete, properly embdedded\footnote{Of course, in saying this, it is implicitly assumed that the surface in question is non-compact.} two-sided stable MOTS $\Sigma$, then $(M,g,k)$ isometrically embeds in the Minkowski spacetime as a space-like slice.
\end{theorem2}

We refer the reader to Section \ref{sec:def} for the precise definition of the class of data mentioned here. In this introduction, we shall limit ourselves to say that this includes as special cases: 
\begin{itemize}
\item{data in harmonic asymptotics (as defined in \cite{EHLS11}), which were proven to be a dense class in general asymptotically flat initial data sets with respect to the topology of weighted Sobolev spaces (see Section 6 of \cite{EHLS11});} 
\item{the $t=constant$ slices, in isotropic coordinates, of the Kerr-Newman spacetime (thus including, as special cases, Schwarzschild, Kerr and Reissner-Nordstr\"om data)  as well as boosts thereof}.
\end{itemize}

It is certainly appropriate to mention here the article \cite{CS14}, which is joint work with R. Schoen, where we show that the rigidity Theorem 2 is essentially sharp by constructing asymptotically flat initial data sets that have large ADM energy and momentum and are \textsl{exactly trivial} outside of a solid cone (of given, yet arbitrarily small opening angle) so that they contain plenty of complete, stable MOTS of planar type. A posteriori, this strongly justifies our requirement that the metric $g$ in the previous statement has some nice asymptotics at infinity. Moreover, such \textsl{flexibility} result turns out to allow the construction of new classes of $N$-body solutions of the Einstein constraint equations. 

For time-symmetric data, namely when $k=0$, marginally outer trapped surfaces are nothing but \textsl{minimal} surfaces and hence Theorem 2 generalizes the following result, of independent interest.

\begin{thm}\cite{Car13}
Let $(M,g)$ be an asymptotically Schwarzschildean 3-manifold of non-negative scalar curvature. If it contains a complete, properly embedded two-sided stable minimal surface $\Sigma$, then $\left(M,g\right)$ is isometric to the Euclidean space $\R^{3}$ and $\Sigma$ is an affine plane. 
\end{thm}

Despite the formal analogy, the proof of Theorem 2 significantly differs from that of its time-symmetric counterpart. Basically, the non-variational nature of MOTS does not allow a direct application of the results concerning isolated singularities of variational problems (see the monograph by L. Simon \cite{Sim85} and reference therein) and thus a more \textsl{ad hoc} argument is needed. As a result, the proof in question, though quite lengthy, has two remarkable advantages: firstly, it is self-contained, and secondly it highlights the key role of \textsl{stability} over the inessential variational structures themselves. Thus, the reader shall find here a substantially different proof of the rigidity result in \cite{Car13}, at least when the ambient dimension equals three.

Lastly, we mention that the latter rigidity result (Theorem 2) can also be interpreted in terms of restrictions on the blow-up set of Jang's equation, based on arguments which go back to the proof of the \textsl{Positive Energy Theorem} by Schoen-Yau \cite{SY81} and we refer the reader to the beautiful survey by Andersson, Eichmair and Metzger \cite{AEM11} for further details on this correspondence.

\

\textsl{Acknowledgments}. The author wishes to express his deepest gratitude to his PhD advisor Richard Schoen for his outstanding guidance, and for the great example he has set. He would also like to thank Otis Chodosh, Michael Eichmair and Greg Galloway for several useful discussions and for pointing out some relevant references, as well as the referee for some relevant and helpful remarks about the first version of the present article.

\section{Definitions and notations}\label{sec:def}

We need to start by recalling a few basic definitions. 

\subsection{Initial data sets}

\begin{definition}\label{def:IDS}
Given an integer $n\geq 3$, we define an \textsl{initial data set} to be a triple $(M,g,k)$ where:
\begin{itemize}
\item{$M$ is a complete $\mathcal{C}^{3}$ manifold of dimension $n$}
\item{$g$ is a $\mathcal{C}^{2}$ Riemannian metric on $M$;}
\item{$k$ is a $\mathcal{C}^{1}$ symmetric $(0,2)$ tensor on $M$.}
\end{itemize}
For an initial data set, we define the \textsl{mass density} $\mu$ and the current density $J$ by means of the equations
\[\begin{cases}
\mu=\frac{1}{2}\left(R_{g}+\left(\textrm{tr}_{g}k\right)^{2}-\left\|k\right\|^{2}_{g}\right) \\
J=\textrm{div}_{g}\left(k-\left(\textrm{tr}_{g}k\right)g\right).
\end{cases}
\]
We say that $(M,g,k)$ satisfies the \textsl{dominant energy condition} (which we often abbreviate as DEC) if at any point of $M$
\[ \mu\geq |J|_{g}.
\]
\end{definition}

When $(M,g,k)$ arises as a space-like slice inside a spacetime $(\mathbb{L},\gamma)$ the densities $\mu$ and $J$ are defined as certain components of the stress-energy tensor $T$ and thus the equations above should be considered as restrictions deriving from the Gauss and Codazzi equations in $(\mathbb{L},\gamma)$, known as \textsl{Einstein constraint equations}.

We now restrict our attention to a special subclass of initial data sets, which are the object of study of the second part of this article.

\begin{definition}\label{def:AFIDS}
Given an integer $n\geq 3$, an initial data set $(M,g,k)$ is called asymptotically flat if:
\begin{enumerate}
\item{there exists a compact set $Z\subset M$ (the \textsl{interior} of the manifold) such that $M\setminus Z$ consists of a disjoint union of \textsl{finitely many} ends, namely $M\setminus Z=\bigsqcup_{l=1}^{N}E_{l}$ and for each index $l$ there exists a smooth diffeomorphism $\Phi_{l}:E_{l}\to\R^{n}\setminus B_{l}$ for some open ball $B_{l}\subset \R^{n}$ containing the origin so that the pulled-back metric $\left(\Phi_{l}^{-1}\right)^{\ast}g$ and the pushed-forward tensor  $\left(\Phi_{l}\right)_{\ast}k$ satisfy the following conditions:
\begin{equation*}
\begin{cases}
\left(\left(\Phi_{l}^{-1}\right)^{\ast}g\right)_{ij}-\delta_{ij}=p_{ij}, \ \ p_{ij}(x)\in O_{2}(|x|^{-(n-2)})\\
\left(\left(\Phi_{l}\right)_{\ast}k\right)^{ij}=O_{1}(|x|^{-(n-1)}) 
\end{cases}
 \ \ \textrm{as} \ |x|\to\infty
\end{equation*}
}
\item{both the \textsl{mass density} $\mu$ and the \textsl{current density} are integrable, namely
\[(\mu, J)\ \in \ \mathcal{L}^{1}(M).
\]
}
\end{enumerate}
In the time-symmetric case, namely when $k=0$, we will simply refer to $(M,g)$ as an \textsl{asymptotically flat manifold}.
\end{definition}

\subsection{Boosted harmonic asymptotics}

For our purposes, it is appropriate to enlarge the class of data under consideration from those in \textsl{harmonic asymptotics} (see \cite{EHLS11}) to its \textsl{closure} under the operation of \textsl{relativistic boost} (inside a given spacetime), namely when a trasformation of the form
\bd
\begin{cases}
\left(x^{0}\right)'=\left(1-\beta^{2}\right)^{-1/2}\left(x^{0}-\beta x^{1}\right)  \\
\left(x^{1}\right)'=\left(1-\beta^{2}\right)^{-1/2}\left(x^{1}-\beta x^{0}\right)  \\
\left(x^{2}\right)'=x^{2} \\
\left(x^{3}\right)'=x^{3}
\end{cases}
\ed
is performed and the resulting $(x^{0})'=0$ space-like slice is considered. Notice that here $0\leq\b<1$ is the \textsl{speed} describing the boost (in normalized unit, with $c=1$).

\begin{definition}\label{def:BDS}
We say that an initial data set $(M,g,k)$ (see Definition \ref{def:AFIDS}) is in \textsl{boosted harmonic asymptotics} if the metric $g$ has the form
\bd
g(x)=\sum_{l=1}^{n}\left(1+\frac{\mathcal{K}\beta_{l}^{2}}{\left|x\right|^{n-2}_{\ast}}\right)dx^{l}\otimes dx^{l}+O_{2}\left(|x|_{\ast}^{-\left(n-1\right)}\right)
\ed
where we have set
\bd
\left|x\right|^{2}_{\ast}=\sum_{l=1}^{n}\z_{l}^{2}\left(x^{l}\right)^{2}
\ed
for some fixed positive real numbers $\b_{1},\ldots, \b_{n}, \z_{1}, \ldots, \z_{n}$ and non-negative $\mathcal{K}$.
\end{definition}

Of course, in the previous definition (for a given such metric $g$) the constant $\mathcal{K}$ is only determined up to a \textsl{positive} scaling factor but since we are only concerned about it being (or not being) equal to zero this turns out to be a convenient choice for our treatment. A motivation for the introduction of the class above is given by the following basic example.

\begin{remark}
The Schwarzschild spacetime is described in the so-called \textsl{isotropic coordinates} (due to Eddington) by 
\[
\gamma=-f(x)dt\otimes dt +\left(1+\frac{\mathcal{M}}{2\left|x\right|}\right)^{4}\del, \ \ f(x)=\left(\frac{1-\frac{\mathcal{M}}{2\left|x\right|}}{1+\frac{\mathcal{M}}{2\left|x\right|}}\right)^{2}
\]
(where $\d$ denotes here the Euclidean metric on $\R^{3}$ (in fact, on $\R^{3}\setminus\left\{\left|x\right|\leq r_{S}\right\}$ for $r_{S}=\mathcal{M}/2$))
and therefore, by restricting to the hypersurface $t=\b x^{1}$ (for some $0\leq\b<1$) we get the space-like metric
\bd
g=\left[\left(1+\frac{\mathcal{M}}{2\left|x\right|}\right)^{4}-\b^{2}f(x)\right]dx^{1}\otimes dx^{1}+\left(1+\frac{\mathcal{M}}{2\left|x\right|}\right)^{4}\sum_{i=2,3}dx^{i}\otimes dx^{i}
\ed
which can be Taylor-expanded as
\bd
g=\left[\left(1-\b^{2}\right)+\left(1+\b^{2}\right)\left(\frac{2\mathcal{M}}{\left|x\right|}\right)\right]dx^{1}\otimes dx^{1}+\left(1+\frac{2\mathcal{M}}{\left|x\right|}\right)\sum_{i=2,3}dx^{i}\otimes dx^{i}+O_{2}\left(\left|x\right|^{-2}\right).
\ed
Therefore, replacing the coordinates $\left\{x\right\}$ by means of \textsl{asymptotically flat} coordinates $\left\{X\right\}$ 
\begin{equation*}
\begin{cases}
X^{1}=\left(1-\b^{2}\right)^{1/2}x^{1} \\
X^{2}=x^{2} \\
X^{3}=x^{3} 
\end{cases}
\end{equation*}
we finally get 
\[
g=\left[1+\left(\frac{1+\b^{2}}{1-\b^{2}}\right)\left(\frac{2\mathcal{M}}{\left|X\right|_{\ast}}\right)\right]dX^{1}\otimes dX^{1}+\left(1+\frac{2\mathcal{M}}{\left|X\right|_{\ast}}\right)\sum_{i=2,3}dX^{i}\otimes dX^{i}+O_{2}\left(\left|X\right|_{\ast}^{-2}\right)
\]
where, in this case
\bd
\left|X\right|_{\ast}=\left(1-\b^{2}\right)^{-1}\left(X^{1}\right)^{2}+\left(X^{2}\right)^{2}+\left(X^{3}\right)^{2}.
\ed
\end{remark}

\subsection{Positive mass theorems and their rigidity statements}

We recall the notions of ADM energy and momentum, which arose in the context of the Hamiltonian formulation of General Relativity \cite{ADM59} and were shown to be well-defined in \cite{Bar86}.

\begin{definition}\label{def:ADM}
Given an asymptotically flat initial data set $(M,g,k)$ (so that both $\mu$ and $\left|J\right|_{g}$ are integrable) one can define the ADM \textsl{energy} $\mathcal{E}$ and the ADM \textsl{momentum} $\mathcal{P}$ \textsl{at each end} to be
\bd
\mathcal{E}=\frac{1}{2\left(n-1\right)\omega_{n-1}}\lim_{r\to\infty}\int_{\left|x\right|=r}\sum_{i,j=1}^{n}\left(g_{ij,i}-g_{ii,j}\right)\nu_{0}^{j}\,d\mathscr{H}^{n-1}
\ed
\bd
\mathcal{P}_{i}=\frac{1}{\left(n-1\right)\omega_{n-1}}\lim_{r\to\infty}\int_{\left|x\right|=r}\sum_{i,j=1}^{n}\pi_{ij}\nu_{0}^{j}\,d\mathscr{H}^{n-1}
\ed
where we have set $\pi=k-\left(\textrm{tr}_{g}k\right)g$ (the \textsl{momentum tensor})\footnote{Of course, indices are raised and lowered using the ambient metric $g$.}, $\nu_{0}^{j}=\frac{x^{j}}{\left|x\right|}$ and $\omega_{n-1}$ is the volume of the standard unit sphere in $\R^{n}$.
\end{definition}

Our second rigidity result is based on the following fundamental theorem.

\begin{theorem}\cite{SY79, SY81, Wit81, Eic13, EHLS11}\label{thm:PMT}
Let $(M,g,k)$ be an asymptotically flat initial data set of dimension $3\leq n< 8$ with one end and satisying the dominant energy condition. Then $\mathcal{E}\geq \left|\mathcal{P}\right|$. Moreover, $\mathcal{E}=0$ if and only if $(M,g,k)$ can be isometrically embedded in the Minkowski spacetime $\left(\mathbb{M},\eta\right)$ as a space-like hypersurface so that $g$ is the induced metric from $\eta$ and $k$ is the second fundamental form (in particular $M$ is topologically $\R^{n}$). In the time-symmetric case $\mathcal{E}\geq 0$ and equality holds if and only if $(M,g)$ is isometric to the Euclidean space $\left(\R^{n},\d\right)$.
\end{theorem}

We remark that the same conclusions also hold true when $M$ has multiple ends, in which case the inequalities are in fact true at the level of each end. \newline
Thanks to this result, we can define the ADM \textsl{mass} to be the norm (with respect to the Minkowski metric $\eta$) of the four-vector $\left(\mathcal{E},\mathcal{P}\right)$ namely
\[
\mathcal{M}=\sqrt{\mathcal{E}-\left|\mathcal{P}\right|^{2}}.
\]
Various remarks are in order. A first proof of this result was given, in the time-symmetric case, by Schoen and Yau \cite{SY79} using minimal surfaces techniques and extended later, via Jang's equation, to show that the \textsl{energy} is non-negative for general asymptotically flat data \cite{SY81}. The statement that the \textsl{mass} is non-negative (namely $\mathcal{E}\geq \left|\mathcal{P}\right|$) was set by Witten in 1981\cite{Wit81}  based on the use of spinors and the Dirac equation and detailed by Parker and Taubes in \cite{PT82}.  
The Schoen-Yau proof of the rigidity statement corresponding to \textsl{null energy} was gotten under the extra technical assumption that if the dimension of $M$ equals three then $\textrm{tr}_{g}(k)\leq C|x|^{-3}$, which was later refined by Eichmair \cite{Eic13} to $\textrm{tr}_{g}(K)\leq C |x|^{-\alpha}$ for some $\alpha>2$. The spinorial approach does not require this assumption and since all 3-manifolds are spin one can in fact state the theorem in the form we gave above. More precisely, for spin manifolds Parker and Taubes proved that if $\mathcal{E}=0$ then the ambient Riemann tensor of the spacetime $\textrm{Rm}=\textrm{Rm}_{\gamma}$ vanishes identically on $M$, and the same conclusion was also obtained \cite{Yip87, BC96, CM06} for the $\mathcal{M}=0$ case. At that stage, one can give self-contained arguments proving that $(M,g,k)$ must isometrically embed inside the Minkowski spacetime $(\mathbb{M},\eta)$ as a space-like slice (see, for instance, \cite{Nar10}).     

\subsection{Marginally outer trapped hypersurfaces}

Let a four dimensional Lorentzian manifold $(\mathbb{L},\gamma)$ be given and let $\left(M,g,k\right)$ be an initial data set inside it. Therefore, if $\upsilon$ is the future directed time-like unit normal vector field to $M$ we will have $k(X,Y)=\gamma\left(D^{\gamma}_{X}\upsilon,Y\right)$ for $X,Y\in\Gamma(TM)$ and $D^{\gamma}$ the Levi-Civita connection of $\gamma$. 
In this setting, let $\Sigma\hookrightarrow M$ be a complete, two-sided surface in $M$: we will denote by $\nu$ a (choice of) smooth unit normal vector field of $\Sigma$ in $M$ and, by convention, we will refer to such choice as outward pointing. At this stage, we can define $l_{+}=\upsilon+\nu$ (resp. $l_{-}=\upsilon-\nu$) as future directed outward (resp. future directed inward) pointing null vector field along $\Sigma$. 
The surface $\Sigma$ is a codimension two submanifold of $\mathbb{L}$ and therefore its extrinsic geometry cannot be described in terms of a scalar function. Instead, it is customary in General Relativity to decompose its second fundamental form into \textsl{two} scalar valued \textsl{null second forms} that will be denoted by $\chi_{+}, \chi_{-}$ and are associated to $l_{+}, l_{-}$ respectively. More precisely, the function $\chi_{+}$ is defined by
\bd
\chi_{+}:T_{p}\Sigma\times T_{p}\Sigma\to\R, \ \ \ \ \chi_{+}(X,Y)=\gamma\left(D^{\gamma}_{X}l_{+}, Y\right)
\ed 
and similarly for $\chi_{-}.$
Furthermore, we consider the associated \textsl{null mean curvatures} that are gotten by tracing with respect to the first fundamental form induced on $\Sigma$ by the metric $g$:
\bd
\theta_{\pm}=\textrm{tr}_{g}\chi_{\pm}=\textrm{div}_{\Sigma}l_{\pm}.
\ed 

A simple, but useful remark is that in fact the null mean curvatures satisfy the equation
\bd
\theta_{\pm}=\textrm{tr}_{\Sigma}k\pm H
\ed
where $H$ denotes the scalar mean curvature of $\Sigma$ in $(M,g)$.
We will limit ourselves to recall that $\theta_{\pm}$ measure the divergence of outgoing and ingoing light rays emanating from $\Sigma$, respectively. In the most trivial example, that of a round sphere in Euclidean slices of the Minkowski spacetime, one obviously has $\theta_{-}<0$ and $\theta_{+}>0$, but in presence of a gravitational field it might happen that for a given surface $\Sigma$ both $\theta_{-}$ and $\theta_{+}$ are negative, in which case we say that $\Sigma$ is a \textsl{trapped surface}. 

\begin{definition}\label{def:MOTS}
Let $(\mathbb{L},\gamma)$ be a four dimensional Lorentzian manifold, let $(M,g,k)$ be an initial data set (see Definition \ref{def:AFIDS}) and let $\Sigma$ be a complete surface in $M$. With the notation above we say that $\Sigma$ is outer trapped if $\theta_{+}<0$ on $\Sigma$. Similarly, we say that a complete surface $\Sigma$ is \textsl{marginally outer trapped} if instead the equation
\bd\label{eq:MOTSeq}
\theta_{+}=0
\ed 
is satisfied.
\end{definition}

Despite their non-variational nature, MOTS do have a suitable notion of \textsl{stability} as suggested by Anderrson, Mars and Simon \cite{AMS08}. For minimal submanifolds (say, for simplicity, of codimension one) the Jacobi operator arises both from the second variation of the area functional and from the (pointwise) first variation of the mean curvature: while the former approach is not applicable to MOTS, the latter can easily be extended. In the setting above, we can consider a normal variation $\left\{\Sigma_{t}\right\}_{t\in\left(-\e,\e\right)}$ of $\Sigma$ in $M$ described by a vector field $X=\phi\nu$ for some compactly supported, smooth function $\phi\in\mathcal{C}^{\infty}_{c}\left(\Sigma,\R\right)$. For $t\in\left(-\e,\e\right)$, a suitably small neighbourhood of 0, let $\nu_{t}$ be the outward normal vector field of $\Sigma_{t}$ in $M$, set $l_{t}=\upsilon+\nu_{t}$ and let $\theta(t)$ be the corresponding null mean curvature. \textsl{Notice that, from now onwards, we will systematically omit the $+$ sign while referring to these quantities.}
It is well-know (see, for instance, Section 2 of \cite{EHLS11}) that the first pointwise variation of the null mean curvature is given by
\bd
\left[\frac{\partial\theta}{\partial t}\right]_{t=0}=\overline{L}(\phi)
\ed
$\overline{L}$ being the operator (for $V=k(\nu,\cdot)_{\Sigma}^{\sharp}$)
\bd
\overline{L}(\phi)=\D_{\Sigma}\phi-2g\left(V,\nabla_{\Sigma}\phi\right)+\left(\left(\mu+J(\nu)\right)+\frac{1}{2}\left|\chi\right|^{2}-K-\textrm{div}_{\Sigma}V+\left|V\right|^{2}\right)\phi.
\ed

It is well-known (see, for instance, \cite{AMS08}) that the operator $\overline{L}$ is not necessarily self-adjoint, yet its principal eigenvalue $\l_{1}(\overline{L})$ is real (this follows from the Krein-Rutman theorem). As a result, it makes sense to give the following definition.

\begin{definition}
In the setting above, we will say that a complete, MOTS is \textsl{stable} if for every regular, relatively compact domain $\Omega$ one has that $\l_{1}(\overline{L},\Omega)\geq0$.
\end{definition}

The operator $\overline{L}$ is significantly more complicated than the Jacobi operator for minimal surfaces and, correspondingly, the associated stability condition is much less useful than the usual stability condition. This obstacle is overcome by introducing the symmetrized operator
\bd
L(\phi)=\D_{\Sigma}\phi+\left(\left(\mu+J\left(\nu\right)\right)+\frac{1}{2}\left|\chi\right|^{2}-K\right)\phi
\ed
for which the following comparison result holds.

\begin{proposition}\label{pro:stabCOMP}\cite{GS06} Let $\Sigma$ be a complete MOTS in an initial data set $\left(M,g,k\right)$, let $\Omega$ be a relatively compact domain in $\Sigma$ and Let $\l_{1}(\overline{L}, \Omega)$ (resp. $\l_{1}(L, \Omega)$) be the principal eigenvalue of the operator $\overline{L}$ (resp. $L$) on $\Omega$. Then
\bd
\l_{1}(L,\Omega)\geq\l_{1}(\overline{L},\Omega).
\ed
\end{proposition}

\
\textbf{Notations.} We denote by $R$ (resp. $Ric(\cdot,\cdot)$) the scalar (resp. Ricci) curvature of $(M,g)$, by $R_{\Sigma}$ (resp. $K$) the scalar (resp. Gaussian) curvature of $\Sigma\hookrightarrow (M,g)$ and by $\nu$ (a choice of) its unit normal. We let $C$ be a real constant which is allowed to vary from line to line, and we specify its functional dependence only when this is relevant.

\section{An extension of Hawking's theorem on the topology of black holes}\label{sec:Hawking}

This section is devoted to the proof of Theorem 1.

\begin{proof}

We give here the proof of part (2) and so let $\varphi:\Sigma\to M$ be a complete, \textsl{non-compact} two-sided immersed\footnote{For the sake of simplicity and uniformity, Theorem 1 was stated for \textsl{embedded} MOTS, but for what concerns part (2) such assumption is unnecessary, so that we provide here the proof for immersions. At the level of regularity, it is enough to assume that the maps $\varphi:\Sigma\to M$ is $\mathcal{C}^3$.} stable MOTS.
Based on the Riemann mapping theorem, the universal cover of $\varphi(\Sigma)$ is conformally equivalent to either $\mathbb{C}$ or the unit disk $\mathbb{D}$. If the latter case happened, we would have a positive solution $w$ on $\mathbb{D}$, endowed with the pull-back metric $h=\pi^{\ast}(\varphi^{\ast}g)$ of the equation
\bd
\D_h w-K_h w+\left(\left(\mu+J(\nu)+\frac{1}{2}\left|\chi\right|^{2}\right)\circ\pi\right) w=0
\ed
just gotten by lifting the function on $\Sigma$ whose existence is guaranteed by Theorem 1 in \cite{FS80} applied to the symmetrized stability operator $L$ (here we are exploiting the comparison result by Galloway and Schoen, Proposition \ref{pro:stabCOMP}). We have denoted by $\pi:\mathbb{D}\to\varphi(\Sigma)$ the covering map. Using the dominant energy condition, we know that $\mu+J(\nu)+\left|\chi\right|^{2}/2\geq 0$, thus the equation above contradicts Corollary 3 in \cite{FS80}. It follows that the universal cover of $\varphi(\Sigma)$ is conformally equivalent to $\mathbb{C}$ and thus $\varphi(\Sigma)$ is conformal either to $\mathbb{C}$ itself or to $\mathbb{A}$.

For what concerns the rigidity assertion, let $\varphi:\Sigma\to M$ be a cylindrical MOTS. We are going to exploit a \textsl{conformal deformation trick} which has its roots in the work by D. Fischer-Colbrie (see e. g. pg. 127 of \cite{FC85}) and that was recently used, in the time-symmetric context, in a joint paper with O. Chodosh and M. Eichmair (see Appendix C of \cite{CCE15}).  
Similarly to what we just did, thanks to the stability assumption, the comparison Proposition \ref{pro:stabCOMP} and Theorem 1 in \cite{FS80} we can find a smooth positive function $u:\Sigma\to\mathbb{R}$ such that $Lu=0$, where $L$ is the symmetrized stability operator of $\Sigma$. We claim that the conformally deformed metric on $\Sigma$ given by $u^2 \varphi^{\ast}g$ is complete. This is a consequence of the argument that proves Theorem 1 in \cite{FC85}, which we can follow almost verbatim modulo replacing the Jacobi operator with the symmetrized stability operator and the assumption that the scalar-curvature is non-negative with the dominant energy condition.
That being said, we observe that the Gauss curvature of such metric is given by the well-known equation
\be\label{eq:gc1}
K_{u^2\varphi^{\ast}g}=u^{-2}\left(K_{\varphi^{\ast}g}+\frac{|\nabla u|^2}{u^2}-\frac{\Delta u}{u}\right)
\ee
and since $Lu=0$ namely $-\Delta_{\varphi^{\ast}g} u= -K_{\varphi^{\ast}g}u+Q u$ we end up finding
\be\label{eq:gc2}
K_{u^2\varphi^{\ast}g}=u^{-2}\left(\frac{|\nabla u|^2}{u^2}+Q\right)
\ee
where of course we have set $Q=\left(\mu+J(\nu)+\frac{1}{2}\left|\chi\right|^{2}\right)\circ\varphi$. Hence $K_{u^2\varphi^{\ast}g}\geq 0$ on the cylinder $\Sigma$. Now, a classic theorem by S. Cohn-Vossen \cite{CV35} (that was later significantly extended by Cheeger-Gromoll \cite{CG72}) ensures that $(\Sigma, u^2\varphi^{\ast}g)$ must be flat, that is to say $K_{u^2\varphi^{\ast}g}=0$ identically on $\Sigma$. The last equation above \eqref{eq:gc2} ensures that $Q=0$ along $\varphi(\Sigma)$ and $u=\textsl{const.}$ and thus \eqref{eq:gc1} gives $K_{\varphi^{\ast}g}=0$ as well. Thereby the proof is complete.
\end{proof}

\section{Isometric embedding in the Minkowski spacetime}\label{sec:embedding}

We now turn to the analysis of asymptotically flat data sets. One first needs to gain a nice description at infinity for a complete stable MOTS without making use of the general results for isolated singularities of geometric variational problems, which are not at disposal for this class of surfaces because of their non-variational nature. We shall prove the following statement, of independent interest.

\begin{proposition}\label{pro:expinf}
Let $(M,g,k)$ be an asymptotically flat initial data set of dimension three and let $\Sigma\hookrightarrow M$ be a complete, properly embedded stable MOTS. Then each end of $\Sigma$ coincides, outside a compact set, with the graph of a function $u\in\mathcal{C}^{2}(\Pi;\mathbb{R})$ for which the following pointwise estimates hold:
\[
|u(x')|+|x'||\nabla u(x')|+|x'|^2|\nabla\nabla u(x')|=e(x') \ \textrm{where} \ e(x')=O(|x'|^{\varepsilon}) \ \forall \ \varepsilon>0
\]	
as  $|x'|\to\infty$.
Here $\left\{x\right\}$ is a set of asymptotically flat coordinates for the corresponding end of the ambient manifold $M$, $x'=(x^1,x^2)$ and $\Pi:=\left\{x^3=0\right\}$ in those coordinates.
\end{proposition}	

Let us recall that in Theorem 2 (and, hence, throughout this section) the surface $\Sigma$ is assumed to be two-sided. Furthermore, at the level of regularity, it suffices to assume that $\Sigma\subset M$ as a $\mathcal{C}^3$ surface.

\begin{remark}
Thanks to the conclusion of Theorem 1 we know that $\Sigma$ is conformally diffeomorphic to either the plane $\mathbb{C}$ or the cylinder $\mathbb{A}$ and hence, if properly embedded, it has respectively one or two ends.	
\end{remark}	

\begin{remark}
It is readily checked that our argument shows that in fact the conclusion above applies to every unbounded connected component of $\Sigma\setminus Z$ provided $\Sigma$ is an embedded stable MOTS.
\end{remark}	

For the sake of conceptual clarity, we shall divide the proof of Proposition \ref{pro:expinf} in a few steps, according to the sequence of lemmata below.

\begin{lemma}\label{lem:blowdown} Every unbounded connected component of $\Sigma\setminus Z$ (for $Z$ the core of $M$) has finite total curvature and quadratic area growth.
	\end{lemma}	
	
\begin{proof}	

This proof follows the arguments given in the first half of Section 3 of \cite{Car13} rather closely, so we shall limit ourselves to sketch it and emphasize the differences, when they occur.
 
Let then $\Sigma^{i}\hookrightarrow E_{j}$ be an unbounded connected component of $\Sigma\setminus Z$ (for $Z$ as in Definition \ref{def:IDS}). 
Thanks to the curvature estimates by Andersson and Metzger \cite{AM10}, the MOTS equation and decay assumption on the momentum tensor we know that for any sequence $\lambda_{m}\searrow 0$ there exists a subsequence (which we do not rename) such that $\lambda_{m}\Sigma^{i}$ converges smoothly in $\mathbb{R}^{3}\setminus\left\{0\right\}$ to a stable minimal lamination $\mathcal{L}$. We have already proven in Section 3 of \cite{Car13} that any such lamination consists of flat planes only, in fact (up to an ambient rotation) $\mathcal{L}=\mathbb{R}^{2}\times Y$ with $Y\subset\mathbb{R}$ closed. This implies that the decay of the second fundamental form of $\Sigma$ can be upgraded to $|A(x)||x|=o(1)$ as $|x|\to\infty$, which in turn is the key to prove that, possibly removing a larger compact set, $\Sigma^{i}$ consists of a finite union of (at most two, by Theorem 1) annular connected components. By this we mean that each such connected component has the topology of a half-cylinder and each large coordinate sphere in the ambient end in question shall intersect that component transversely along a circle. Thus, possibly enlarging the core $Z$ we can certainly assume we had started with one of those annular components, that is $\Sigma^{i}$ itself.
Finally, arguing by contradiction by means of a blow-down procedure one can exploit the Coarea formula to show that $\Sigma^{i}$ has quadratic area growth and, thus, use the stability inequality together with the logarithmic cut-off trick to check that the total curvature of $\Sigma^{i}$ has to be finite. 
\end{proof}

Our next goal is to exploit all of this information in order to improve the curvature decay of $\Sigma$, namely to prove that in fact $|A(x)|\leq C|x|^{-1-\alpha}$ for some $\alpha>0$ and, at that stage, we shall get the conclusion of Proposition \ref{pro:expinf} in a fairly direct way. 

In order to proceed in the argument, we start by observing that the symmetrized stability inequality for $\Sigma$ implies that given any positive $\rho$
\[ \left(1-\frac{\rho}{2}\right)\int_{\Sigma}|A|^{2}\xi^{2}\,d\mathscr{H}^{2}\leq\int_{\Sigma}|\nabla_{\Sigma}\xi|^{2}\,d\mathscr{H}^{2}+\frac{C}{\rho}\int_{\Sigma}\frac{1}{1+d(p,p_{0})^{3}}\xi^{2}\,d\mathscr{H}^{2}
\]
for any compactly supported function $\xi$ of class $\mathcal{C}^{1}$. 

\textsl{We let from now onwards be $\Sigma^{i}$ an annular connected component of $\Sigma\setminus Z$ (based on the above discussion) and $\Sigma^{i}_{0}\hookrightarrow (\mathbb{R}^{3},\delta)$ be the corresponding submanifold in the Euclidean ambient. For simplicity of notation, we shall simply denote it by $\Sigma_0$.}

By the usual comparison relation between $A$ and $A_{0}$, namely
\begin{equation}\label{eq:stdcomp}
|A(x)-A_0(x)|_g\leq \frac{C}{|x|^2}(|x||A(x)|_g+1)
\end{equation}
 one can deduce that in fact
\begin{equation}\label{eq:rearr} (1-\rho)\int_{\Sigma_{0}}|A_{0}|^{2}\xi^{2}\,d\mathscr{H}^{2}\leq\int_{\Sigma_{0}}|\nabla_{\Sigma_{0}}\xi|^{2}\,d\mathscr{H}^{2}+\frac{C}{\rho}\int_{\Sigma_{0}}|x|^{-3}\xi^{2}\,d\mathscr{H}^{2}.
\end{equation}
Here $\left\{x\right\}$ is a set of asymptotically flat coordinates for $M$ along the end in question and of course in the last inequality we are referring to the two-dimensional Hausdorff measure in $(\mathbb{R}^{3},\delta)$. Without loss of generality $\Sigma_{0}\subset\mathbb{R}^{3}\setminus B_{r_{0}}$ with $\partial\Sigma_{0}\subset\partial B_{r_{0}}$ and, correspondingly, the test function $\xi$ is required to be compactly supported in $\Sigma_{0}\setminus \overline{B}_{r_{0}}.$
For reasons that will be clear soon in the proof of the following lemma, we set from now onwards $\rho=\frac{1}{6}=\frac{1}{3(n-1)}$ since $n=3$.

\begin{lemma}\label{lem:key}
	There exists a constant $C>0$  which only depends on $(M,g,k)$ such that for all functions $\varphi$ that are compactly supported and vanish in a neighbourhood of $\partial B_{r_{0}}\hookrightarrow \R^{3}$ we have
	\bd
	\int_{\Sigma_{0}}\left|A_{0}\right|^{2}\varphi^{2}\,d\mathscr{H}^{2}\leq C\int_{\Sigma_{0}}\left(1-\left(\overline{\nu}\cdot\nu_{0}\right)^{2}\right)\left|\nabla_{\Sigma_{0}}\varphi\right|^{2}\,d\mathscr{H}^{2}+C\int_{\Sigma_{0}}\left|x\right|^{-3}\varphi^{2}\,d\mathscr{H}^{2}
	\ed
	where $\overline{\nu}\in S^{2}$ is any constant unit vector of unit length.
\end{lemma} 

\begin{proof}
	As a first step, let us set $\xi=\varphi\left(1-\left(\overline{\nu}\cdot\nu_{0}\right)^{2}\right)^{1/2}$ in the stability inequality \eqref{eq:rearr}, with $\varphi\in\mathcal{C}^{1}$ and compactly supported away from $\partial\Sigma_{0}$: such function $\xi$ is not $\mathcal{C}^{1}$ itself, but locally Lipschitz (hence $\mathscr{H}^{2}-$a.e. differentiable) with $\left|\nabla_{\Sigma_{0}}\xi\right|\leq \left|A_{0}\right|$ and a standard approximation argument justifies its use in \eqref{eq:rearr}. \textsl{In this whole proof we will use $\nabla$ in place of $\nabla_{\Sigma_{0}}$ and $\Delta$ in place of $\Delta_{\Sigma_{0}}$ in order to make the estimates more readable}.
	Correspondingly, expanding all terms on the right-hand side we get
	
	\begin{align*}
	(1-\rho) \int_{\Sigma_{0}}\left|A_{0}\right|^{2}\left(1-\left(\overline{\nu}\cdot\nu_{0}\right)^{2}\right)\varphi^{2}\,d\mathscr{H}^{2}&\leq\int_{\Sigma_{0}}\left[\left(1-\left(\overline{\nu}\cdot\nu_{0}\right)^{2}\right)\left|\nabla\varphi\right|^{2}+\varphi^{2}\left|\nabla\left(1-\left(\overline{\nu}\cdot\nu_{0}\right)^{2}\right)^{1/2}\right|^{2}\right]\,d\mathscr{H}^{2} \\
	& +2\int_{\Sigma_{0}}\varphi\left(1-\left(\overline{\nu}\cdot\nu_{0}\right)^{2}\right)^{1/2}\nabla\varphi\cdot\nabla\left(1-\left(\overline{\nu}\cdot\nu_{0}\right)^{2}\right)^{1/2}\,d\mathscr{H}^{2} \\
	&+C\int_{\Sigma_{0}}\left|x\right|^{-3}\varphi^{2}\,d\mathscr{H}^{2}.
	\end{align*}
	At that stage, we need to write the second summand on the right-hand side in a more useful way. Integration by parts gives
	\[\int_{\Sigma_{0}}\nabla\varphi^{2}\cdot\nabla\left(1-(\overline{\nu}\cdot\nu_{0})^{2}\right)\,d\mathscr{H}^{2}=\int_{\Sigma}\varphi^{2}\Delta(\overline{\nu}\cdot\nu_{0})^{2}\,d\mathscr{H}^{2}
	\]
	and since (due to the Codazzi equation)
	\[ \Delta(\overline{\nu}\cdot\nu_{0})^{2}=2\left|\nabla\nu_{0}\cdot\overline{\nu}\right|^{2}-2|A_{0}|^{2}(\overline{\nu}\cdot\nu_{0})^{2}+2\sum_{i}\tau_i(H_{0})(\tau_{i}\cdot\overline{\nu})(\overline{\nu}\cdot\nu_{0})
	\]
	we come up with the functional inequality
	\begin{align*}
	(1-\rho)\int_{\Sigma_{0}}|A_{0}|^{2}\varphi^{2}\,d\mathscr{H}^{2} & \leq C\int_{\Sigma_{0}}(1-(\overline{\nu}\cdot\nu_{0})^{2})|\nabla\varphi|^{2}\,d\mathscr{H}^{2} \\
	& +\int_{\Sigma_{0}}\varphi^{2}\left(|\nabla\nu_{0}\cdot\overline{\nu}|^{2}+\left|\nabla\left(1-\left(\overline{\nu}\cdot\nu_{0}\right)^{2}\right)^{1/2}\right|^{2}\right)\,d\mathscr{H}^{2}\\
	&+\int_{\Sigma_{0}}\varphi^{2}\sum_{i}\tau_i (H_{0})(\tau_{i}\cdot\overline{\nu})(\overline{\nu}\cdot\nu_{0})\,d\mathscr{H}^{2} +C\int_{\Sigma_{0}}|x|^{-3}\varphi^{2}\,d\mathscr{H}^{2}.
	\end{align*}
	Notice that here and above $\left\{\tau_{i}\right\}$ is just a local orthonormal basis for the tangent space to $\Sigma_{0}$. Now, since in fact
	\[
	|\nabla\nu_{0}\cdot\overline{\nu}|^{2}+\left|\nabla\left(1-\left(\overline{\nu}\cdot\nu_{0}\right)^{2}\right)^{1/2}\right|^{2}=\frac{|\nabla\nu_{0}\cdot\overline{\nu}|^{2}}{1-(\nu_{0}\cdot\overline{\nu})^{2}}
	\]
	we can follow, almost verbatim, the proof of Lemma 1 in \cite{SS81} in order to get the pointwise geometric inequality
	\[ |A_{0}|^{2}-\frac{|\nabla\nu_{0}\cdot\overline{\nu}|^{2}}{1-(\nu_{0}\cdot\overline{\nu})^{2}}\geq\frac{1}{n-1}|A_{0}|^{2}-\frac{2}{n-1}|A_{0}||H_{0}|
	\]
	(that in our case we specifiy with $n=3$) which implies
	\begin{align*}\frac{1}{3}\int_{\Sigma_{0}}|A_{0}|^{2}\varphi^{2}\,d\mathscr{H}^{2} & \leq C\int_{\Sigma_{0}}(1-(\overline{\nu}\cdot\nu_{0})^{2})|\nabla\varphi|^{2}\,d\mathscr{H}^{2} \\
	& +\int_{\Sigma_{0}}\varphi^{2}\sum_{i}\tau_i(H_{0})(\tau_{i}\cdot\overline{\nu})(\overline{\nu}\cdot\nu_{0})\,d\mathscr{H}^{2}+\int_{\Sigma_{0}}|A_{0}||H_{0}|\,d\mathscr{H}^{2} \\
	& +C\int_{\Sigma_{0}}|x|^{-3}\varphi^{2}\,d\mathscr{H}^{2}.
	\end{align*} 
	Integrating by parts the second summand on the right-hand side of the previous inequality and applying the usual algebraic manipulations to absorb the terms involving the second fundamental form $A_{0}$ on the left-hand side (exploiting the MOTS equation and equation \eqref{eq:stdcomp} to handle the summands of the form $\int_{\Sigma_0}|H_0|^2\,d\mathscr{H}^2$) we conclude
	\bd
	\int_{\Sigma_{0}}\left|A_{0}\right|^{2}\varphi^{2}\,d\mathscr{H}^{2}\leq C\int_{\Sigma_{0}}\left(1-\left(\overline{\nu}\cdot\nu_{0}\right)^{2}\right)\left|\nabla_{\Sigma_{0}}\varphi\right|^{2}\,d\mathscr{H}^{2}+C\int_{\Sigma_{0}}\left|x\right|^{-3}\varphi^{2}\,d\mathscr{H}^{2}
	\ed
	which is precisely the inequality we were supposed to prove.
\end{proof}

As this point, the strategy is to combine this improved inequality with a Poincar\'e-type inequality in order to prove that the outer total curvature $\int_{\Sigma_{0}\setminus B_{\sigma}}|A_{0}|^{2}\,d\mathscr{H}^{2}$ decays like a negative power of $\sigma$ as we let $\sigma$ go to infinity. At that stage, this integral estimate will be turned into a pointwise estimate by means of the De Giorgi lemma.

As a preliminary remark, we observe that the improved decay estimate $|A_0(x)|\leq o(1)|x|^{-1}$ (which follows from the proof of Lemma \ref{lem:blowdown} together with \eqref{eq:stdcomp})  implies via a standard graphicality lemma (as in Chapter 2 of \cite{CM11}) that for any $\sigma$ large enough $\Sigma_0$ can be described, in the Euclidean annulus of radii $\sigma$ and $2\sigma$ as a graph over a coordinate plane. Specifically, for any such $\sigma$ and there exists a plane $\Pi=\Pi^{(\sigma)}$ in coordinates $\left\{x\right\}$ and a suitably smooth function $v=v^{(\sigma)}:\Pi\to\mathbb{R}$ whose graph coincides with $\Sigma_0$ in the aforementioned ambient annulus.

\begin{lemma}\label{deca} Let $\Sigma_{0}\hookrightarrow\left(\R^{3},\d\right)$ be as above. Then there exist constants $\alpha>0$ and $C$ such that
	\bd
	J(\sigma)=\int_{\Sigma_{0}\setminus B_{\sigma}}\left|A_{0}\right|^{2}\,d\mathscr{H}^{2}\leq C \sigma^{-2\alpha}.
	\ed
\end{lemma}

\begin{proof}
	Thanks to Lemma A.1 in \cite{Car13} it suffices to show that there exist two constants $\theta\in (0,1)$ and $\xi>0$ such that
	\[ J(2\sigma)\leq\theta J(\sigma)+\xi\sigma^{-1}.
	\]
	Given any $\sigma> r_{0}$, we would like to consider the improved stability inequality (Lemma \ref{lem:key}) with $\varphi_{\sigma}$ a $\mathcal{C}^{1}$ function which vanishes in $B_{\sigma}$, is equal to one outside of $B_{2\sigma}$ and increases linearly for $\sigma\leq r\leq 2\sigma$. Obviously, any such function is not compactly supported, yet this choice can easily be justified considering a suitable sequence of functions monotonically increasing to $\varphi_{\sigma}$ and applying, once again, a logarithmic cut-off trick. This strongly makes use of the conclusion we got in Lemma \ref{lem:blowdown}, namely quadratic area growth and finiteness of the total curvature. The details are rather standard and we omit them here.
	As a result, we obtain
	\bd
	\int_{\Sigma_{0}\setminus{B_{2\sigma}}}\left|A_{0}\right|^{2}\,d\mathscr{H}^{2}\leq 2C\sigma^{-2}\int_{\Sigma_{0}\cap\left(B_{2\sigma}\setminus B_{\sigma}\right)}\left(1-\left(\overline{\nu}\cdot\nu_{0}\right)^{2}\right)\,d\mathscr{H}^{2}+C\sigma^{-1}
	\ed
	and our strategy now is to combine it with a Poincar\'e-Wirtinger inequality.
	
	We can find an upper bound on the first term on the right-hand side as follows: since trivially
	\bd
	1-\left(\overline{\nu}\cdot\nu_{0}\right)^{2}=\frac{1}{4}\left|\overline{\nu}-\nu_{0}\right|^{2}\left|\overline{\nu}+\nu_{0}\right|^{2}\leq\left|\overline{\nu}-\nu_{0}\right|^{2}
	\ed
	we have
	\bd\sigma^{-2}\int_{\Sigma_{0}\cap\left(B_{2\sigma}\setminus B_{\sigma}\right)}\left(1-\left(\overline{\nu}\cdot\nu_{0}\right)^{2}\right)\,d\mathscr{H}^{2}\leq
	\sigma^{-2}\int_{\Sigma_{0}\cap\left(B_{2\sigma}\setminus B_{\sigma}\right)}\left|\overline{\nu}-\nu_{0}\right|^{2}\,d\mathscr{H}^{2}.
	\ed
	 Let then $\nu^{(\sigma)}_{0}$ be the average of $\nu_{0}$ on such $\Sigma_0\cap B_{2\sigma}\setminus B_{\sigma}$: clearly $\nu^{(\sigma)}_{0}$ does not need to be a unit vector in general, but still the following pointwise inequality holds
	\bd
	\left|\nu_{0}-\frac{\nu^{(\sigma)}_{0}}{\left|\nu^{(\sigma)}_{0}\right|}\right|\leq \left|\nu_{0}-\nu^{(\sigma)}_{0}\right|+\left|\nu^{(\sigma)}_{0}-\frac{\nu^{(\sigma)}_{0}}{\left|\nu^{(\sigma)}_{0}\right|}\right|\leq 2\left|\nu_{0}-\nu^{(\sigma)}_{0}\right|
	\ed
	and therefore (letting $\overline{\nu}=\nu^{(\sigma)}_{0}/|\nu^{(\sigma)}_{0}|$)
	\bd
	\sigma^{-2}\int_{\Sigma_0\cap B_{2\sigma}\setminus B_{\sigma}}\left|\overline{\nu}-\nu_{0}\right|^{2}\,d\mathscr{H}^{2}\leq 4\sigma^{-2}\int_{\Sigma_0\cap B_{2\sigma}\setminus B_{\sigma}}\left|\nu_{0}-\nu^{(\sigma)}_{0}\right|^{2}\,d\mathscr{H}^{2}.
	\ed
	By the Area Formula we can rewrite the previous integral as
	\bd
	\int_{\Sigma_0\cap B_{2\sigma}\setminus B_{\sigma}}\left|\nu_{0}-\nu^{(\sigma)}_{0}\right|^{2}\,d\mathscr{H}^{2}=\int_{\textrm{proj}\left(\Sigma_0\cap B_{2\sigma}\setminus B_\sigma\right)}\left|\nu_{0}\circ v-\nu^{(\sigma)}_{0}\right|^{2}\textrm{Jac}(v)\,d\mathscr{L}^{2}
	\ed
	 where $\textrm{proj}:\Sigma_0\to\Pi$ is the Euclidean orthogonal projection and clearly for the Jacobian $\textrm{Jac}\left(v\right)=\sqrt{1+\left|\nabla v\right|^{2}}$. Thanks to the locally uniform bounds for these graphical components (again: as in Chapter 2 of \cite{CM11}), one easily gets that
	\bd
	\int_{\textrm{proj}\left(\Sigma_0\cap B_{2\sigma}\setminus B_{\sigma}\right)}\left|\nu_{0}\circ v-\nu^{(\sigma)}_{0}\right|^{2}\textrm{Jac}(v)\,d\mathscr{L}^{2}\leq C\int_{\textrm{proj}\left(\Sigma_0\cap B_{2\sigma}\setminus B_\sigma\right)}\left|\nu_{0}\circ v-\nu^{(\sigma)}_{0}\right|^{2}\,d\mathscr{L}^{2}
	\ed
	\bd\leq C\sigma^{2} \int_{\textrm{proj}\left(\Sigma_0\cap B_{2\sigma}\setminus B_\sigma\right)}\left|\nabla_{\Sigma_{0}}\nu_{0}\circ v\right|^{2}\,d\mathscr{L}^{2}
\leq C\sigma^2\int_{\Sigma_0\cap B_{2\sigma}\setminus B_{\sigma}}\left|A_{0}\right|^{2}\,d\mathscr{H}^{2}
	\ed
	and hence combining all these equations we come to the final estimate
	\bd
	\int_{\Sigma_{0}\setminus{B_{2\sigma}}}\left|A_{0}\right|^{2}\,d\mathscr{H}^{2}\leq C \int_{\Sigma_{0}\cap B_{2\sigma}\setminus B_{\sigma}}\left|A_{0}\right|^{2}\,d\mathscr{H}^{2}+C\sigma^{-1}.
	\ed
	This is nothing but
	\bd
	J(2\sigma)\leq C\left(J\left(\sigma\right)-J\left(2\sigma\right)\right)+C\sigma^{-1}
	\ed
	or, equivalently
	\bd
	J(2\sigma)\leq\frac{C}{1+C}J(\sigma)+\frac{C}{1+C}\sigma^{-1}
	\ed
	and our claim follows by setting $\theta=\xi=C\left(1+C\right)^{-1}$.
\end{proof}

At this point, we want to turn the previous integral estimate into an improved pointwise estimate. To that aim, we need an adaptation of one basic fact of the De Giorgi-Nash theory, the subsolution estimate, which is discussed in Appendix \ref{sec:dgn}. In order to apply Proposition \ref{dgn}, we also recall in Appendix \ref{sec:sim} a general Simons' type inequality for surfaces in $(\mathbb{R}^{3},\delta)$.

We shall now make use of these results in order to complete the proof of Proposition \ref{pro:expinf}.

\begin{proof}
	Thanks to Lemma \ref{simons}, the MOTS equation and the Schoen-type decay estimate by Andersson-Metzger we get
	\bd
	\D_{\Sigma_{0}}\left|A_{0}(x)\right|^{2}\geq -\frac{C}{\left|x\right|^{5}}-6\left|A_{0}(x)\right|^{4}, \ \ \forall \ \left|x\right|=r>2 r_{\ast}
	\ed
	It follows at once, by trivial manipulations, that one can choose a \textsl{positive} constant $\overline{C}>0$ \textsl{independent of $\overline{x}$} (and where $\overline{r}=\left|\overline{x}\right|/2$) such that the function
	\bd
	u=\overline{r}^{-5/2}+\left|A_{0}\right|^{2}
	\ed
	satisfies a functional inequality of the form \eqref{ellpb}, specifically (for any $\varepsilon>0$)
	\bd
	\D_{\Sigma_{0}}u\geq -\overline{C}\left(\overline{r}^{-5/2}+\left|A_{0}\right|^{2}\right)u, \ x\in B_{\overline{r}}\left(\overline{x}\right).
	\ed
	As a result, we are in position to apply our De Giorgi-Nash inequality, Proposition \ref{dgn}, to the function $u$, for $p=1$ and $\theta=1/2$ thus obtaining (via the integral estimate Lemma \ref{deca}) for $\overline{r}>r_{\ast}$
	\bd
	\sup_{B^{\Sigma_{0}}_{\overline{r}/2}\left(\overline{x}\right)}\left|A_{0}\right|\leq C\left(\frac{1}{\overline{r}^{2}}\int_{B_{\overline{r}}\left(\overline{x}\right)}\left(\overline{r}^{-5/2}+\left|A_{0}\right|^{2}\right)\,d\mathscr{H}^{2}\right)^{1/2}\leq  C \overline{r}^{-1-\min\left\{1/4, \alpha\right\}}.
	\ed
	
	The crucial remark here is that the constant $C$ in the final estimate can be chosen independently of $\overline{x}$ (to greater extend of $\overline{r}$) because (based on the statement of Theorem \ref{dgn}) the quantity
	\bd
	\overline{r}^{2\left(1-\frac{1}{1+\e}\right)}\left(\int_{B^{\Sigma}_{\overline{r}}\left(\overline{x}\right)}\left(\overline{r}^{-5/2}+\left|A_{0}\right|^{2}\right)^{1+\e}\,d\mathscr{H}^{2}\right)^{1/\left(1+\e\right)}
	\ed
	is uniformly bounded as we let $\overline{r}\to \infty$, this being true in fact for \textsl{any} positive value of $\e$.
	As a result we conclude that
	\bd
	\sup_{B^{\Sigma_{0}}_{\overline{r}/2}\left(\overline{x}\right)}\left|A_{0}\right|\leq C\overline{r}^{-1-\alpha'}, \ \alpha'=\min\left\{1/4, \alpha\right\}
	\ed
	and therefore, as a special case
	\bd
	\left|A_{0}\left(\overline{x}\right)\right|\leq C\left|\overline{x}\right|^{-1-\alpha'}, \ \forall \ \overline{x}\in\Sigma\setminus B_{2r_{\ast}}.
	\ed

This improved decay estimate on the second fundamental form implies at once (due to its radial integrability) that the tangent cone to $\Sigma_{0}$ at infinity is unique and hence, possibly taking a smooth extension inside a compact set, we can assume that there exist $\Pi$ and $u\in\mathcal{C}^{2}\left(\Pi;\R\right)$ such that $\Sigma_{0}$ coincides with its graph, at least outside of a suitably large ball.
Furthermore, we have that 
\[
\nabla u(x')=O(\left|x'\right|^{-\alpha'}), \ \ \nabla\nabla u(x')=O(\left|x'\right|^{-1-\alpha'}).
\]
 At this stage a bootstrap argument (along the lines of the one presented in Appendix A of \cite{Car13}) based on linear PDE theory in $\mathbb{R}^{2}$ allows to improve the decay rate, hereby completing the proof. 
\end{proof}

Lastly, we are now in position to give the proof of Theorem 2 and deduce that an initial data set having boosted harmonic asymptotics and containing a properly embedded stable MOTS must isometrically embed in $(\mathbb{M},\eta)$ as a space-like slice. 

\begin{proof}

Because of the stability comparison theorem by Galloway-Schoen (Proposition \ref{pro:stabCOMP}) we know that for any test function $\phi\in\mathcal{W}^{1,2}(\Sigma)$ the following functional inequality is satisfied:
\bd\int_{\Sigma}\left[\mu+J(\nu)+\frac{1}{2}\left|\chi\right|^{2}\right]\phi^{2}\,d\mathscr{H}^{2}\leq\int_{\Sigma}|\nabla_{\Sigma}\phi|^{2}\,d\mathscr{H}^{2}+\int_{\Sigma}K\phi^{2}\,d\mathscr{H}^{2}.
\ed
The conclusion of Proposition \ref{pro:expinf}, concerning the structure at infinity of $\Sigma$, together with the well-known result by Shiohama \cite{Shi85} concerning the Gauss-Bonnet theorem for open manifolds give that 
\bd\int_{\Sigma}K=2\pi\left[\chi(\Sigma)-N'\right]
\ed
for $N'$ the total number of ends of $\Sigma$. Applying the logarithmic cut-off trick to our inequality (which is legitimate because of Lemma \ref{lem:blowdown}) and combining it with the previous equation, we must conclude that $\Sigma\simeq\mathbb{R}^{2}$, and that $\mu+J(\nu)=0, \ \chi=0$ identically on $\Sigma$. That being said, one can follow once more the argument by Fischer-Colbrie and Schoen to get to the conclusion that $\Sigma$ has to be intrinsecally flat, namely its Gauss curvature is zero at every point.

\
Let us denote by $y^{1}, y^{2}$ Euclidean coordinates on $\Pi$ and let them be completed to an asymptotically flat set of coordinates $\left\{y\right\}$ for $\R^{3}$. Also, let $C\in SO(3)$ be the (Euclidean) isometry relating $\left\{x\right\}$ and $\left\{y\right\}$, so that the tangent vectors to $\Sigma$ have $\left\{x\right\}$-coordinates given by $\left(w_{l}\right)^{i}=c^{i}_{j}\left(v_{l}\right)^{j}$ for $l=1,2$
where clearly
\begin{equation*}
v_{1}=\begin{pmatrix} 1 \\
0 \\
\partial_{y^{1}}u
 \end{pmatrix}, \ \
v_{2}=\begin{pmatrix} 0 \\
1 \\
\partial_{y^{2}}u
 \end{pmatrix}. \ \
\end{equation*}

The metric $\overline{g}$ induced on $\Sigma$ by the ambient metric $g$ has, in terms of the matrix $C$ (and using the decay properties of $u$) an asymptotic expansion of the form
\bd
g\left(v_{i}, v_{j}\right)=\overline{g}_{ij}=\d_{ij}+\frac{\mathcal{K}}{r(y)}\omega_{ij}+O_{2}(\left|r(y)\right|^{-2+2\varepsilon})
\ed 
with $r^{2}(y)=\sum_{i=1}^{2}\z_{i}^{2}\left(y^{i}\right)^{2}+ \z_{3}^{2}u^{2}(y)$, $\omega_{ij}=\sum_{l=1}^{3}\left(c_{i}^{l}\right)\left(c_{j}^{l}\right)\b_{l}^{2}$ and $\varepsilon\in (0,1/2)$ is a constant that we can take as small as we wish. In particular, let us emphasize that $\omega_{ii}>0$ for any choice of the index $i$ due to the fact that $C\in SO(3)$.
We can easily determine the Christhoffel symbols of $\overline{g}$ 
\bd
\overline{\Gamma}_{ij}^{k}=-\frac{\mathcal{K}}{2}\sum_{p=1}^{2}\d^{kp}\left(\frac{\omega_{ip}\z_{j}^{2}y^{j}+\omega_{jp}\z_{i}^{2}y^{i}-\omega_{ij}\z_{p}^{2}y^{p}}{r^{3}(y)}\right)+O_{1}(\left|r(y)\right|^{-3+2\varepsilon})
\ed
and thus one can differentiate further and get
\begin{align*}
\overline{\Gamma}_{ij,l}^{k} & =-\frac{\mathcal{K}}{2}\sum_{p=1}^{2}\d^{kp}\left[\frac{\omega_{ip}\z_{j}^{2}\d^{j}_{l}+\omega_{jp}\z_{i}^{2}\d^{i}_{l}-\omega_{ij}\z_{p}^{2}\d^{p}_{l}}{r^{3}\left(y\right)}-3\frac{\omega_{ip}\z^{2}_{j}\z^{2}_{l}y^{j}y^{l}+\omega_{jp}\z^{2}_{i}\z^{2}_{l}y^{i}y^{l}-\omega_{ij}\z^{2}_{p}\z^{2}_{l}y^{p}y^{l}}{r^{5}(y)}\right] \\ 
& +O\left(\left|r(y)\right|^{-4+2\varepsilon}\right).
\end{align*}
It follows that the expression of the scalar curvature of $\Sigma$ is given, in these coordinates, by
\bd
R_{\Sigma}=\overline{g}^{ij}\left(\Gamma_{ij,k}^{k}-\Gamma_{ik,j}^{k}\right)+O(\left|r(y)\right|^{-4+2\varepsilon})
\ed
\bd
=-\mathcal{K}\sum_{i\neq k}\left[-\frac{\omega_{ii}\z_{k}^{2}}{r^{3}\left(y\right)}-3\frac{\omega_{ik}\z^{2}_{i}\z^{2}_{k}y^{i}y^{k}-\omega_{ii}\z_{k}^{4}(y^{k})^{2}}{r^{5}\left(y\right)}\right]+O(\left|r(y)\right|^{-4+2\varepsilon})
\ed
so that we can conveniently rewrite it in the final form
\bd
R_{\Sigma}=-\frac{\mathcal{K}}{r^{3}\left(y\right)}\left[-\omega_{11}\z_{2}^{2}-\omega_{22}\z_{1}^{2}-\frac{3}{r^{2}(y)}\left(2\omega_{12}\z_{1}^{2}\z_{2}^{2}y^{1}y^{2}-\omega_{11}\left(\z_{2}^{2}y^{2}\right)^{2}-\omega_{22}\left(\z_{1}^{2}y^{1}\right)^{2}\right)\right]+O(\left|r(y)\right|^{-4+2\varepsilon}).
\ed
Now, we know that $R_{\Sigma}$ is identically equal to zero and therefore this is true, as a special case, on the coordinate line where $y^{2}=0$: the expansion of the scalar curvature along that path is given by
\bd
R_{\Sigma}=-\frac{\mathcal{K}}{r^{3}(y)}\left[-\omega_{11}\z_{2}^{2}-\omega_{22}\z_{1}^{2}+\frac{3\omega_{22}\left(\z_{1}^{2}y^{1}\right)^{2}}{\left(\z_{1}y^{1}\right)^{2}+\left(\z_{3}u(y)\right)^{2}}\right]+O(\left|y\right|^{-4+2\varepsilon})
\ed
so \textsl{if} $\mathcal{K}$ \textsl{were not zero} by letting $|y|\to\infty$ we would get the algebraic condition $2\omega_{22}\z_{1}^{2}=\omega_{11}\z_{2}^{2}$. Considering, symmetrically, the coordinate line where $y^{1}=0$ we would be led to the system
\begin{equation*}
\begin{cases}
2\omega_{11}\z_{2}^{2}=\omega_{22}\z_{1}^{2} \\
2\omega_{22}\z_{1}^{2}=\omega_{11}\z_{2}^{2}
\end{cases}
\end{equation*}
and hence, by comparison, we would get the conclusion $\omega_{11}\z_{2}^{2}=\omega_{22}\z_{1}^{2}=0$, contradiction. Thus we necessarily have $\mathcal{K}=0$ and then (keeping in mind Definition \ref{def:BDS}) this implies that the ADM \textsl{energy} of the metric $g$ of the initial data set $(M,g,k)$ is zero and so, thanks to the rigidity statement in Theorem \ref{thm:PMT}, this forces $(M,g,k)$ to isometrically embed as a space-like slice in the Minkowski model $(\mathbb{M},\eta)$, which is what we had to prove.

\end{proof}

\appendix

\section{A De Giorgi-Nash estimate}\label{sec:dgn}

We shall state and briefly discuss here an almost immediate adaptation of a fundamental result by De Giorgi and Nash to complete surfaces in the Euclidean space that are not necessarily minimal.

\begin{proposition}\label{dgn}
	Let $\Sigma_{0}$ be as in Section \ref{sec:embedding}.
	For $x_{0}\in\Sigma_{0}$ and $r>2r_{0}$, let $B_{r}^{\Sigma_{0}}\left(x_{0}\right)$ be the \textsl{intrinsic} ball of center $x_{0}$ with $\left|x_{0}\right|=2r$ and radius $r$ on $\Sigma_{0}$.
	Suppose $u\in \mathcal{W}^{1,2}\left(B_{r}^{\Sigma_{0}}\left(x_{0}\right)\right)$ is non-negative, locally bounded and weakly satisfies
	\be\label{ellpb}
	\Delta_{\Sigma_{0}}u+au\geq 0
	\ee 
	and that there exists $\e>0$ such that
	\bd
	r^{2-\frac{2}{1+\e}}\left\|a\right\|_{L^{1+\e}\left(B_{r}^{\Sigma_{0}}\left(x_{0}\right)\right)}\leq C'.
	\ed
	Then there exists $r_{\ast}$ such that when $r>r_{\ast}$ the following statement holds: for every $\theta\in\left(0,1\right)$ and $p>0$ there exists a constant $C$ such that
	\bd
	\sup_{B_{\theta r\left(x_{0}\right)}^{\Sigma_{0}}} u\leq C\left(r^{-2}\int_{B^{\Sigma_{0}}_{r}\left(x_{0}\right)}u^{p}\,d\mathscr{H}^{2}\right)^{1/p}
	\ed
	where $C=C\left(\theta,p,C'\right)$.
\end{proposition}

Let us describe how the general well-known proof for Euclidean balls can be adapted to our setting. The argument to prove Theorem \ref{dgn} when $\Omega\subset\R^{d}$ a bounded regular domain, is based on the Sobolev inequality
\be\label{sobG}
\left(\int_{\Omega}\left|u\right|^{\frac{dp}{d-p}}\,d\mathscr{L}^{d}\right)^{\frac{d-p}{d}}\leq C(d,p)\int_{\Omega}\left|\nabla u\right|^{p}\,d\mathscr{L}^{d}
\ee
for $u\in \mathcal{W}^{1,p}_{0}\left(\Omega\right)$.
In turn, this general version can easily be deduced from the case $p=1$ namely
\be\label{sobS}
\left(\int_{\Omega}\left|u\right|^{\frac{d}{d-1}}\,d\mathscr{L}^{d}\right)^{\frac{d-1}{d}}\leq C(d)\int_{\Omega}\left|\nabla u\right|\,d\mathscr{L}^{d}
\ee
by replacing the function $u$ by $\left|u\right|^{\frac{\left(d-1\right)p}{\left(d-p\right)}}$ and recalling the basic fact that $D\left|u\right|=\left(\textrm{sgn}u\right)\left|Du\right|$ for $\mathscr{L}^{d}-$a.e. point $x\in\Omega\subset\R^{d}$.
When $\Sigma_{0}\hookrightarrow \left(\R^{3},\d\right)$ is a \textsl{minimal} surface and $\Omega=B^{\Sigma_{0}}_{r}\left(x\right)$ then inequalities like \eqref{sobS} (and hence \eqref{sobG}) still hold true, while if $\Sigma_{0}$ is only known to be, say, a smooth submanifold \eqref{sobS} with locally bounded mean curvature then they should be replaced by the Michael-Simon inequality (see \cite{MS73} and \cite{HS74} for this extended version):
\bd
\left(\int_{B^{\Sigma_{0}}_{r}\left(x\right)}\left|u\right|^{2}\,d\mathscr{H}^{2}\right)^{1/2}\leq C\int_{B^{\Sigma_{0}}_{r}\left(x\right)}\left(\left|\nabla_{\Sigma_{0}} u\right|+\left|H_{0}\right|\left|u\right|\right)\,d\mathscr{H}^{2}.
\ed
But notice that, by applying the Cauchy-Schwarz inequality on the second summand of the right-hand side we get
\bd
\left(\int_{B^{\Sigma_{0}}_{r}\left(x\right)}\left|u\right|^{2}\,d\mathscr{H}^{2}\right)^{1/2}\leq C\left[\int_{B^{\Sigma_{0}}_{r}\left(x\right)}\left|\nabla_{\Sigma_{0}} u\right|\,d\mathscr{H}^{2}+\left(\int_{B^{\Sigma_{0}}_{r}\left(x\right)}\left|H_{0}\right|^{2}\,d\mathscr{H}^{2}\right)^{1/2}\left(\int_{B^{\Sigma_{0}}_{r}\left(x\right)}\left|u\right|^{2}\,d\mathscr{H}^{2}\right)^{1/2}\right]
\ed
and hence, thanks to the MOTS equation satisfied by $\Sigma_{0}$ and the usual comparison relations for $H, H_{0}$ we  can find $r_{\ast}$ so that
\bd
\left(\int_{B^{\Sigma_{0}}_{r}\left(x\right)}\left|u\right|^{2}\,d\mathscr{H}^{2}\right)^{1/2}\leq C\int_{B^{\Sigma_{0}}_{r}\left(x\right)}\left|\nabla_{\Sigma_{0}} u\right|\,d\mathscr{H}^{2}
\ed
whenever $\left|x\right|>2r_{\ast}$ and $u\in \mathcal{W}^{1,1}_{0}\left(B^{\Sigma_{0}}_{r}\left(x\right)\right)$.
Therefore, we can deduce \eqref{sobG} from this and at that point follow, with very minor variations, the standard Euclidean proof of Theorem \ref{dgn} (see, for instance Theorem 5.3.1 in \cite{Mor66}).

\section{A Simons' inequality  for general surfaces}\label{sec:sim}

\begin{lemma}\label{simons}
	Let $\Sigma_{0}\hookrightarrow (\mathbb{R}^{3},\delta)$ any immersed surface. Then
	\bd
	\D_{\Sigma_{0}}\left|A_{0}\right|^{2}\geq -2\left|\nabla^{2}_{\Sigma_{0}}H_{0}\right|\left|A_{0}\right|-6\left|A_{0}\right|^{4}+2\left|\nabla_{\Sigma_{0}}A_{0}\right|^{2}.
	\ed
\end{lemma}

\begin{proof}
	The proof of this Lemma is a variation on the well-known argument by J. Simons. Indeed, working with a local basis $\left\{\tau_{1},\tau_{2}\right\}$ we get by the Gauss and Codazzi equations the identity
	\bd
	a_{ik,jk}=a_{ik,kj}+\sum_{m}\left(a_{ki}a_{jm}-a_{ji}a_{km}\right)a_{mk}+\sum_{m}\left(a_{kk}a_{jm}-a_{jk}a_{km}\right)a_{mi}
	\ed
	and hence
	\bd
	\frac{1}{2}\D_{\Sigma_{0}}\left|A_{0}\right|^{2}=\sum_{i,j}a_{ij}\D_{\Sigma_{0}}a_{ij}+\sum_{i,j}\left|\nabla_{\Sigma_{0}}a_{ij}\right|^{2} =\sum_{i,j,k}a_{ij}a_{ij,kk}+\sum_{i,j,k}a^{2}_{ij,k}=\sum_{i,j,k}a_{ij}a_{ik,jk}+\sum_{i,j,k}a^{2}_{ij,k}
	\ed
	\bd
	=\sum_{i,j,k}a_{ij}a_{kk,ij}+\sum_{i,j,k,m}a_{ij}\left(a_{ki}a_{jm}-a_{ji}a_{km}\right)a_{mk}+\sum_{i,j,k,m}a_{ij}\left(a_{kk}a_{jm}-a_{jk}a_{km}\right)a_{mi}+\sum_{i,j,k}a^{2}_{ij,k}
	\ed
	\bd
	\geq -\left|\nabla_{\Sigma_{0}}\nabla_{\Sigma_{0}}H_{0}\right|\left|A_{0}\right|-(1+\sqrt{2})\left|A_{0}\right|^{4}+\left|\nabla_{\Sigma_{0}}A_{0}\right|^{2}
	\ed
	where in the last step we have used the Cauchy-Schwarz inequality and cancelled out two summands that are patently equal (modulo renaming the indices). The claim follows at once.
\end{proof}
      
\bibliographystyle{plain}

\end{document}